\documentclass[reqno]{article}
\usepackage{amsmath,amsfonts,amssymb,amsthm}

\newtheorem{theorem}{Theorem}
\newtheorem{lemma}[theorem]{Lemma}

\newtheorem{corollary}[theorem]{Corollary}

\theoremstyle{definition}

\newtheorem*{ack}{Acknowledgement}

\newcommand\noproof{\hfill$\Box$}

\newcommand\Bi{{\mathrm{Bi}}}
\newcommand\CBi{{\mathrm{CBi}}}
\newcommand\eps{\varepsilon}
\newcommand\la{\lambda}
\newcommand\las{\lambda_*}

\renewcommand\Pr{{\mathbb P}}
\newcommand\E{{\mathbb E}}
\newcommand\Var{{\mathrm{Var}}}
\newcommand\dto{\overset{\mathrm{d}}{\to}}

\newcommand\tX{{\widetilde X}}

\newcommand\op{o_{\mathrm{p}}}
\newcommand\Op{O_{\mathrm{p}}}
\newcommand\cF{\mathcal{F}}
\newcommand\cA{\mathcal{A}}
\newcommand\cB{\mathcal{B}}
\newcommand\bb[1]{\bigl(#1\bigr)}

\begin{document}
\title{Asymptotic normality of the size of the giant component via a random walk}
\author{B\'ela Bollob\'as%
\thanks{Department of Pure Mathematics and Mathematical Statistics,
Wilberforce Road, Cambridge CB3 0WB, UK and
Department of Mathematical Sciences, University of Memphis, Memphis TN 38152, USA.
E-mail: {\tt b.bollobas@dpmms.cam.ac.uk}.}
\thanks{Research supported in part by NSF grants CNS-0721983, CCF-0728928
and DMS-0906634, and ARO grant W911NF-06-1-0076}
\and Oliver Riordan%
\thanks{Mathematical Institute, University of Oxford, 24--29 St Giles', Oxford OX1 3LB, UK and
Department of Mathematical Sciences, University of Memphis, Memphis TN 38152, USA.
E-mail: {\tt riordan@maths.ox.ac.uk}.}}
\date{October 21, 2010; revised April 14, 2011}
\maketitle

\begin{abstract}
In this paper we give a simple new proof of a result of Pittel and Wormald concerning
the asymptotic value and (suitably rescaled) limiting distribution
of the number of vertices in the giant component of $G(n,p)$ above
the scaling window of the phase transition. Nachmias and Peres used
martingale arguments to study Karp's exploration process, obtaining
a simple proof of a weak form of this result. We
use slightly different martingale arguments to obtain a much
sharper result with little extra work.
\end{abstract}

\section{Introduction and results}

The component of a random graph containing a given vertex may
be `explored' by a step-by-step process that is by now well known,
described in detail below. A key feature of this process
is that vertices are `examined' one at a time, and tested
for edges to `new' vertices. This means that the behaviour
of the exploration is closely connected to that of a certain random walk.
In the context of random graphs, this process
was introduced by Karp~\cite{Karp} in 1990; 
slightly earlier, Martin-L\"of~\cite{ML86}
used essentially the same process in a different context,
namely the study of epidemics, where it arises even more naturally.
Somewhat later, Aldous~\cite{Aldous} introduced
a variant of the process adapted to explore {\em all} components
of a random graph; 
recently, analyzing this latter exploration with
martingale techniques related to those in~\cite{ML86},
Nachmias and Peres~\cite{NP_giant} gave a simple proof that in the weakly supercritical range,
i.e., when $p=(1+\eps)/n$ where $\eps=\eps(n)$ satisfies
$\eps\to 0$ but $\eps^3n\to \infty$,
the largest component of $G(n,p)$ contains $2\eps n+\op(\eps n)$ vertices.
(They also studied the weakly subcritical case, which we shall not discuss
further here.)

Here we shall analyze the same process more carefully, obtaining a simple
new proof of the following asymptotic normality result due to Pittel and Wormald~\cite{PWio}.
Let $\rho=\rho_\la$ denote the survival probability of the Galton--Watson
branching process in which the number of offspring of each individual has a Poisson
distribution with mean $\la$.
For $\la>1$ we may write $\rho_\la$ as the unique positive solution to
\begin{equation}\label{rdef}
 1-\rho=e^{-\la\rho}.
\end{equation}
When $\la>1$ we write $\las$ for $\la(1-\rho_\la)$; this is often known as the {\em dual
branching process parameter} to $\la$, and satisfies $\las<1$ and $\las e^{-\las}=\la e^{-\la}$.
(The corresponding Poisson branching process provides an approximation
of the random graph in the vicinity of a generic vertex outside the giant component.)

\begin{theorem}\label{th_gen}
Let $p=\la/n$ where $\la=\la(n)$ satisfies $\la=O(1)$ and $(\la-1)^3n\to\infty$
as $n\to\infty$, and
let $L_1$ denote the number of vertices in the largest component
of $G(n,p)$.
Then
\begin{equation*}
 \frac{L_1 - \rho n}{\sigma} \dto N(0,1),
\end{equation*}
where $\dto$ denotes convergence in distribution, $N(0,1)$ is the standard
normal distribution,
$\rho=\rho_\la>0$ is defined by \eqref{rdef},
and
\[
 \sigma^2 = \frac{\rho(1-\rho)}{(1-\las)^2} n.
\]
\end{theorem}

The special case of this result in which $\la$ is constant 
goes back to Stepanov~\cite{Stepanov_conn} (see also Pittel~\cite{Pittel_census});
the form above is due to Pittel and Wormald~\cite{PWio}, who proved much
more, including asymptotic joint normality of the sizes of the largest
component and of its 2-core.

Specializing to the barely supercritical case, the
formulae above simplify considerably. Indeed,
it is easy to check that if $\la=1+\eps$ and
$\eps\to 0$, then $\rho_\la=2\eps+O(\eps^2)$, and $\las=1-\eps+O(\eps^2)$.
Thus Theorem~\ref{th_gen} has the following corollary.

\begin{corollary}\label{th1}
Let $\eps=\eps(n)$ satisfy $\eps\to 0$ and $\eps^3 n\to\infty$,
and let $L_1$ denote the number of vertices in the largest component
of $G(n,(1+\eps)/n)$.
Then
\begin{equation}\label{aim}
 \frac{L_1 - \rho n}{\sqrt{2\eps^{-1}n}} \dto N(0,1),
\end{equation}
where $\rho>0$ is defined by \eqref{rdef} with $\la=1+\eps$.
\noproof
\end{corollary}

Under the conditions of Corollary~\ref{th1}
we have $\rho\sim 2\eps$, while the standard deviation $\sqrt{2\eps^{-1}n}$
is $o(\eps n)$, so Corollary~\ref{th1} implies in particular the result
of Nachmias and Peres~\cite{NP_giant} mentioned earlier.

\section{The proof}

We consider the component exploration process as in~\cite{NP_giant},
itself based on those of Karp~\cite{Karp}, Martin-L\"of~\cite{ML86} and Aldous~\cite{Aldous},
although we shall use slightly different terminology and initial conditions.
At each step, every vertex will
have one of three states, {\em active}, {\em explored}, or {\em unseen}.
The exploration will take place in $n$ steps, at times $t=1,\ldots,n$,
starting from the initial state at time $0$, when every vertex is unseen.

Fix an order on the vertices. At step $1\le t\le n$ (i.e., going from time $t-1$ to time $t$)
let $v_t$ be the first active vertex, if there are any;
otherwise $v_t$ is the first unseen vertex. In the latter case
we say that we `start a new component' at step $t$. Having defined $v_t$,
reveal all edges from $v_t$ to (other) unseen vertices; let $\eta_t$
be the number of such edges, and label the corresponding
neighbours of $v_t$ as active; label $v_t$ itself as explored.
After $t$ steps of the process, exactly $t$ vertices have been explored.
We write $A_t$ and $U_t$ for the numbers of active
and unseen vertices after $0\le t\le n$ steps, so $U_t=n-t-A_t$,
$A_0=0$ and $U_0=n$.

After $n$ steps,
it is very easy to see that the process has revealed a spanning forest
in $G$, having first revealed a spanning tree of one component, then a spanning
tree of another component (if there is more than one), and so on.

Write $C_t$ for the number of components started by time $t$,
and set $X_t=A_t-C_t$. We claim that
\begin{equation}\label{claim}
 X_t = A_t-C_t = \sum_{i=1}^t (\eta_i-1).
\end{equation}
Indeed, if in step $t$ we do not start a new component, then we explore an active
vertex and then change $\eta_t$ vertices from unseen to active,
so $A_t-A_{t-1}=\eta_t-1$ and $C_t=C_{t-1}$. If we do start a new component, which happens
if and only if $A_{t-1}=0$, then we explore an unseen vertex, so $A_t-A_{t-1}=A_t=\eta_t$
and $C_t-C_{t-1}=1$.
This establishes \eqref{claim}.

Let $0=t_0<t_1<t_2<\cdots<t_k=n$ enumerate $\{t:A_t=0\}$, i.e., the set of times
at which there are no active vertices. We start exploring the $i$th component
at time $t_{i-1}+1$ and finish at time $t_i$, so
\begin{equation}\label{L1}
 L_1 = \max\{t_i-t_{i-1} : 1\le i\le k\}.
\end{equation}
Since $C_t=i$ for $t_{i-1}<t\le t_i$, recalling that $X_t=A_t-C_t$ we have
\begin{equation}\label{ti}
 t_i = \inf\{t:X_t =-i\}.
\end{equation}
Writing $c(G)$ for the number of components of $G=G(n,p)$, note that $X_n=-c(G)$,
and that $X_t$ may decrease by at most one at each step, so the infimum is defined
for all $1\le i\le c(G)$.

Let $\cF_t$ denote the sigma-field generated by
$\eta_1,\ldots,\eta_t$; in other words, $\cF_t$ is the (finite, of
course) sigma-field generated by all information revealed by step $t$.
Set $U_t'=U_t$ if $A_t>0$ and $U_t'=U_t-1$
otherwise. Then $U_t'$ is the number of edges tested at step $t+1$.
Hence, given $\cF_t$, the random variable $\eta_{t+1}$ has a binomial distribution
with parameters $U_t'$ and $p$:
\[
 \Pr( \eta_{t+1} =k \mid \cF_t) = \binom{U_t'}{k}p^k(1-p)^{U_t'-k}.
\]
If we know the sequence $(\eta_t)$, then we know the entire outcome
of the process, and in particular $L_1$. More precisely, we can use \eqref{claim} 
to find $(X_t)$, then \eqref{ti} to find the $t_i$ (and thus $(C_t)$, $(A_t)$ and $(U_t)$),
and finally \eqref{L1} gives us $L_1$.

So far we have been following (with minor modifications) the definitions
and initial analysis in~\cite{NP_giant}. But now our analysis takes a different route.

Let us write
$D_t$ for the expectation of $\eta_t-1$ given $\cF_{t-1}$, noting that
$D_t$ is random, and satisfies
\[
 D_{t+1} = \E( \eta_{t+1}-1 \mid \cF_t ) =  p U_t'-1.
\]
Recalling that $U_t=n-t-A_t=n-t-X_t-C_t$, and noting that $U_t'=U_t-(C_{t+1}-C_t)$,
this gives
\begin{equation}\label{dt}
 D_{t+1} = p(n-t-X_t-C_{t+1})-1.
\end{equation}

Our next aim is to approximate the process $(X_t)$ that we wish to study
by a simpler process $(\tX_t)$, consisting of a deterministic term plus a term
closely related to a martingale.
Let $\Delta_t = \eta_t-1-D_t$, so $\E(\Delta_t\mid \cF_{t-1})=0$ by the definition of $D_t$.
From \eqref{claim}, \eqref{dt} and $\eta_{t+1}-1=D_{t+1}+\Delta_{t+1}$ we obtain the recurrence
\begin{equation}\label{Xrec}
 X_{t+1} = (1-p) X_t + \Delta_{t+1} + p(n-t)-1-p C_{t+1}.
\end{equation}
Let
\[
 x_t = n-t-n(1-p)^t,
\]
so $x_0=0$ and
\begin{equation}\label{tXr}
x_{t+1} = (1-p)x_t + p(n-t) -1.
\end{equation}
Subtracting \eqref{tXr} form \eqref{Xrec} we see that
\[
 X_{t+1}-x_{t+1} = (1-p) (X_t-x_t) + \Delta_{t+1} - pC_{t+1},
\]
whence
\begin{equation}\label{Xx}
 X_t-x_t = \sum_{i=1}^t (1-p)^{t-i}(\Delta_i-p C_i).
\end{equation}
With this in mind, we define our approximating process by
\begin{equation}\label{tdef}
 \tX_t = x_t + \sum_{i=1}^t (1-p)^{t-i}\Delta_i.
\end{equation}

\begin{lemma}\label{Wc}
For any $p>0$ and any $1\le t\le n$ we have
\[
 | X_t-\tX_t | \le pt C_t.
\]
\end{lemma}
\begin{proof}
From \eqref{Xx} and \eqref{tdef} we have
\[ 
 X_t-\tX_t = -\sum_{i=1}^t (1-p)^{t-i}p C_i.
\]
The result follows immediately since there are $t$ terms in the sum,
each bounded by $p C_t$.
\end{proof}

Let
\[ 
 S_t=\sum_{i=1}^t (1-p)^{-i}\Delta_i,
\]
so $(S_t)$ is a martingale, and
\begin{equation}\label{td2}
 \tX_t = x_t + (1-p)^t S_t.
\end{equation}
As we shall see below, it is easy to obtain very precise results
about the distribution of $(\tX_t)$; before turning to the details,
let us indicate in rather vague terms why this should be the case.

The variance of each $\Delta_i$ is $O(1)$, so $S_t$ and hence $(1-p)^t
S_t$ have variance $O(t)$ and size $\Op(\sqrt t)$.  It is true that
the distribution of $\Delta_t$ depends on earlier values of $X_i$ in a
way that is hard to evaluate exactly, but the dependence is weak: the
conditional variance of $\Delta_t$ is simply $p(1-p)U_{t-1}'$, so if
we can bound the earlier $X_i$ within an additive error of $o(n)$,
then we obtain a bound on the variance of $\Delta_t$ accurate to
within a factor $1+o(1)$. This gives only a $\op(\sqrt t)$ additive
error in the martingale term, which is negligible compared to the
random variation. (It will turn out that we hit the giant component
before seeing many other components, so additional $ptC_t$ error
from Lemma~\ref{Wc} will be negligible.)
This strongly suggests
that given that Theorem~\ref{th_gen} is true, there should be a simple
proof based on the analysis of $(\tX_t)$. As we shall see, this is indeed
the case.

From now on we assume that $p=\la/n$, where $\la=\la(n)>1$ is bounded.
More explicitly, we assume $\la<M$ for some constant $M$.
Often, we write $\la=1+\eps$; we assume also that $\eps^3 n\to\infty$.

For the moment, we study $(\tX_t)$.
Let us first start with a standard
observation; the second part is a special
case of Doob's maximal inequality~\cite[Ch. III, Theorem 2.1]{Doob}.

\begin{lemma}\label{lmax}
Let $(Z_t)_0^\infty$ be a discrete-time martingale with
filtration $(\cF_t)$ and mean $Z_0=0$. Write
$I_t$ for the increment $Z_t-Z_{t-1}$.
Then
\begin{equation}\label{vzt}
 \Var(Z_t) = \sum_{i=1}^t \Var(I_i) = \sum_{i=1}^t \E\bb{\Var(I_i\mid \cF_{i-1})},
\end{equation}
and for any $M\ge 0$,
\[
 \Pr( \max_{i\le t}|Z_i| \ge M) \le \Var(Z_t)/M^2.
\]
\end{lemma}
\begin{proof}
For the first statement, observe that $\E I_i=0$ for all $i$ and $\E Z_t=0$,
while for $i<j$ we have $\E(I_iI_j) = \E( \E(I_iI_j\mid \cF_{j-1})) = \E(0)=0$.
Hence $\Var(Z_t)=\E Z_t^2 =\E\bigl(\sum_1^t I_i\bigr)^2 = \sum_i \E I_i^2 = \sum_i\Var(I_i)$. Also,
$\E(\Var(I_i\mid \cF_{i-1})) = \E(\E(I_i^2\mid \cF_{i-1}))= \E I_i^2$,
proving \eqref{vzt}.

For the second statement, apply Doob's maximal inequality.
Alternatively, simply modify the martingale if $|Z_i|\ge M$
holds for any $i$: let $T$ be the (random) first such $i$,
or $T=t$ if there is no such $i$, 
and set $Z_j'=Z_j$ for $j\le T$ and $Z_j'=Z_T$ for $j> T$.
Since $T$ is a stopping time, the conditional distribution
of $I_i'=Z_i'-Z_{i-1}'$ given $\cF_{i-1}$ is either
the same as that of $I_i$, or zero, so the conditional variances
of the $I_i'$ are at most those of the $I_i$.
Hence, by \eqref{vzt}, $\Var(Z_t')\le \Var(Z_t)$.
Since $\max_{i\le t}|Z_i|\ge M$ if and only if $|Z_t'|\ge M$,
applying Chebyshev's inequality gives the result.
\end{proof}

Let us write $\CBi(m,p)$ for the {\em centered binomial distribution}
obtained by subtracting $mp$ from a random variable with binomial 
distribution $\Bi(m,p)$. Note that the variance of this distribution
is $mp(1-p)$. The conditional distribution of $\Delta_t$ given
$\cF_{t-1}$ is exactly that of a centered binomial
$\CBi(U_{t-1}',p)$. (Previously, we first subtracted one, and then centered,
but of course this is the same as centering directly.)
It follows that the differences $I_i=S_i-S_{i-1}=(1-p)^{-i}\Delta_i$ satisfy
\begin{equation}\label{var}
 \Var(I_i\mid \cF_{i-1}) = (1-p)^{-2i} U_{i-1}'p(1-p),
\end{equation}
so
\[
 \Var(I_i\mid \cF_{i-1}) \le (1-p)^{-2n} np \le (1-M/n)^{-2n}M=O(1).
\]
For any (deterministic) function $t=t(n)$, Lemma~\ref{lmax} thus gives
\begin{equation}\label{smax}
 \sup_{i\le t}|S_i|=\Op(\sqrt{t}).
\end{equation}

Let $f(t)=f_n(t)=n-t-ne^{-pt}$ be the continuous-time form of the
idealized trajectory of $(\tX_t)$ (and hence of $(X_t)$).
It is easy to check that $|f(t)-x_t|=O(1)$,
uniformly in $p\le M/n$ and $0\le t\le n$;
our next lemma shows that $(\tX_t)$ remains close to $f_n(t)$.
\begin{lemma}\label{l2}
For any  $1\le t=t(n)\le n$ we have
\begin{equation}\nonumber
 \sup_{i\le t} |\tX_t-f_n(t)| =\Op(\sqrt{t}).
\end{equation}
\end{lemma}
\begin{proof}
Immediate from \eqref{smax}, \eqref{td2}  and $|f_n(t)-x_t|=O(1)$.
\end{proof}
Together, Lemmas~\ref{Wc} and~\ref{l2} show that $(X_t)$ remains
close to the idealized trajectory $f(t)$, as long as $C_t$ is not too large.
As in~\cite{NP_giant}, the basic idea is now to consider
the solution $t_1=\rho n$ to $f(t_1)=0$, and choose
a suitable $t_0$. We shall show that in the interval $[t_0,t_1-t_0]$
the function $f(t)$ is far enough away from zero that $X_t$ remains
positive, so no new component is started in this interval. 
Then we consider more precisely the time when $X_t$ crosses
below its previous minimum level and use \eqref{ti} to obtain Theorem~\ref{th_gen}.

We start by examining $f$.
Note that
\begin{equation}\label{ft}
 f'(t) = -1+npe^{-pt} = p(n-t-f(t))-1,
\end{equation}
and that $f''(t) = -np^2e^{-pt}$ is negative and uniformly bounded by $M^2/n$.
Since $f'(0)=np-1=\eps$,
it follows that if $t\le \eps n/(2M^2)$, then $f'(t)\ge \eps/2$
and, integrating, that
\begin{equation}\label{fte}
 f(t)\ge \eps t/2.\
\end{equation}

From now on let us pick a function $\omega=\omega(n)$ tending to infinity slowly,
in particular with $\omega^6=o(\eps^3 n)$.
Set
\[ 
 \sigma_0=\sqrt{\eps n}
\]
and
\[
 t_0=\omega\sigma_0/\eps,
\]
ignoring, as usual, the irrelevant rounding to integers. Note for later that $t_0=o(\eps n)$.

\begin{lemma}\label{l1}
Let $Z=-\inf\{X_t:t\le t_0\}$ denote the number of components completely explored
by time $t_0$, and let $T_0=\inf\{t:X_t=-Z\}$ be the time at which
we finish exploring the last such component.
Then $Z\le \sigma_0/\omega$ and $T_0\le \sigma_0/(\eps\omega)$ hold whp.
\end{lemma}
Considering the initial trajectory of the process $(X_t)$, it is not hard
to check that in fact $Z=\Op(\eps^{-1})$ and $T_0=\Op(\eps^{-2})$,
but the weaker bounds above suffice.
\begin{proof}
Let $k=\sigma_0/\omega$. Note that by choice of $\omega$ we have
$k/\sqrt{t_0}\to\infty$.
Let $\cA$ denote the event that $\sup_{t\le t_0}|\tX_t-f(t)| < k/2$. Then by Lemma~\ref{l2},
$\cA$ holds whp.

At time $T_0$ we have $X_{T_0}=-Z$.
Noting that $pt_0=o(1)$, we have $pt_0\le 1/2$ if $n$ is large enough,
which we assume from now on.  Since $T_0\le t_0$ by definition,
it follows that $pT_0\le 1/2$.
But then Lemma~\ref{Wc} gives
\[
 |X_{T_0}-\tX_{T_0}|\le pT_0C_{T_0}\le Z/2,
\]
and thus $\tX_{T_0}\le -Z/2$. Since $f(t)\ge 0$ for $t\le t_0<\rho n$, this gives
$|\tX_{T_0}-f(T_0)|\ge Z/2$. Hence, whenever $\cA$ holds, we have $Z\le k$,
and the first statement follows.

Turning to second statement, recall from \eqref{fte} that $f(t)\ge \eps t/2$ for $t\le t_0=o(\eps n)$.
Consider the interval $I=[\sigma_0/(\eps\omega),t_0]$. In this interval
we have $f(t)\ge \sigma_0/(2\omega)=k/2$, so if $\cA$ holds
then $\tX_t> 0$ for all $t\in I$.
As shown above, we have $\tX_{T_0}\le -Z/2\le 0$, so whenever $\cA$ holds then $T_0\notin I$.
Since $T_0\le t_0$ by definition, this completes the proof.
\end{proof}

Let $T_1=\inf\{t:X_t=-Z-1\}$. Then by the properties of the exploration process,
there is a component with $T_1-T_0$ vertices; we aim to show that
this component has size close to the anticipated size of the giant component.

Since $np=O(1)$, by Lemmas~\ref{Wc} and \ref{l1} we have that
\begin{equation}\label{XXtclose}
  \sup_{t\le T_1} |X_t-\tX_t| \le \sigma_0/\sqrt{\omega}
\end{equation}
holds whp.

Let $t_1=\rho n$, noting that $t_1\sim 2\eps n$ if $\eps\to 0$, and that
$t_1$ is the unique positive solution to $f(t)=0$.
Let $t_1^-=t_1-t_0$ and $t_1^+=t_1+t_0$. Note that $t_1^+=O(\eps n)=O(\sigma_0^2)$.
From \eqref{XXtclose} and Lemma~\ref{l2} we have that
\begin{equation}\label{close}
 \sup_{t\le \min\{T_1,t_1^+\}} |X_t-f(t)| \le \sqrt{\omega}\sigma_0
\end{equation}
holds whp.

Let $a=-f'(t_1)$, so from \eqref{ft} and the definition of $t_1$ we have
\[
 a=-f'(t_1)=1-p(n-t_1) = 1-\la(1-\rho) = 1-\las,
\]
where $\las$ is the dual branching process parameter to $\la$.
In particular,  $a=\Theta(\eps)$.
Since $f(t_1)=0$ and $f''(t)$ is uniformly $O(1/n)$,
recalling that $t_0=o(\eps n)$ 
it follows easily that $f(t_1^-)$ and $f(t_1^+)$ are
both of order $\eps t_0 = \omega\sigma_0$. To be concrete, if
$n$ is large enough, then we certainly have
\[
 f(t_1^-) \ge 10\sqrt{\omega}\sigma_0 \hbox{\quad and\quad} f(t_1^+)\le -10\sqrt{\omega}\sigma_0,
\]
say.
Since $f(t_0)\ge \eps t_0/2\ge 10\sqrt{\omega}\sigma_0$
and $f$ is unimodal, we have $\inf_{t_0\le t\le t_1^-} f(t)\ge 10\sqrt{\omega}\sigma_0$.
Let $\cB$ denote the event described in \eqref{close}.
Then, whenever $\cB$ holds, we have $X_t\ge 0$ for $t_0\le t\le \min\{T_1,t_1^-\}$.
Since $X_{T_1}\le -Z-1<0$, this implies $T_1> t_1^-$.

Recall from Lemma~\ref{l1} that (crudely) $Z\le \sigma_0$ whp. Suppose  $Z\le \sigma_0$,
$\cB$ holds, and $T_1>t_1^+$.
Then from $\cB$ and the bound on $f(t_1^+)$ we have $X_{t_1^+}\le -9\sqrt{\omega}\sigma_0<-Z$,
contradicting $T_1> t_1^+$. It follows that $T_1\le t_1^+$ holds whp.

At this point we have shown that $|T_1-t_1|\le t_0$ holds whp, which gives $|T_1-T_0-t_1|\le 2t_0$.
Since $\omega$ may tend to infinity arbitrarily slowly, this already shows that
$T_1-T_0=t_1+\Op(\sigma_0/\eps)=\rho n+\Op(\sqrt{\eps^{-1}n})$. To go further, we
next analyze the distribution
of $X_{t_1}$ more precisely.

From Lemma~\ref{l1} and the bound $T_1>t_1^-$ whp just proved,
whp we have $C_{t_1^-}=Z\le \sigma_0/\omega$.
Noting that $t_0=t_1-t_1^-=o(n)$, it follows that $\E C_{t_1}=o(n)$.
Lemma~\ref{Wc} and Lemma~\ref{l2} thus give $|X_t-f(t)|=\op(n)$,
uniformly in $t\le t_1$.
Since $X_t-f(t)$ is deterministically bounded by $n$, it
follows that $\E|X_t-f(t)|$ and hence $\E |X_t+C_{t+1}-f(t)|$ are $o(n)$,
uniformly in $t\le t_1$. Let $u_t=n-t-f(t)=ne^{-pt}$. Since $U_t'=n-t-(X_t+C_{t+1})$, we have
shown that
\begin{equation}\label{diff}
 \E \sum_{t=0}^{t_1-1} |U_t'-u_t| = o(t_1n) = o(\eps n^2).
\end{equation}

Note that
\begin{multline}
 p(1-p)\sum_{t=0}^{t_1-1} (1-p)^{-2t}u_t
 \sim p\sum_{t=0}^{t_1-1}e^{2pt}ne^{-pt} \\
 \sim n^2p \int_0^{\rho} e^{\la x}{\rm d}x
 = n\la\la^{-1}(e^{\la\rho}-1) = n\rho/(1-\rho), \label{vsum}
\end{multline}
using $e^{-\la\rho}=1-\rho$ in the last step.
\begin{lemma}\label{norm}
The distribution of $S_{t_1}$ is asymptotically normal with mean 0
and variance $n\rho/(1-\rho)$.
\end{lemma}
\begin{proof}
Recall that $(S_t)$ is a martingale with $S_0=0$, and that the conditional
distribution of the $i$th difference $(1-p)^{-i}\Delta_i$ is $(1-p)^{-i}$ times
a centered binomial $\CBi(U_{i-1}',p)$, and has conditional variance given by \eqref{var}.
The result follows easily by a standard martingale central limit theorem 
such Brown~\cite[Theorem 2]{Brown}.
Note that here the differences are not uniformly bounded. However,
we can write $\Delta_i$ as the sum of a random number $U_{i-1}'$ of $\CBi(1,p)$ random variables,
plus $n-U_{i-1}'$ zero variables. We can take the new variables multiplied by $(1-p)^{-i}$
as the differences
of a martingale $(S_j')$ with the property that $S_t=S_{nt}'$.
In this way we obtain a martingale with the same (random) final value in which 
the differences are bounded by $(1-p)^{-n}=O(1)$.
The (random) sum of the (old or new) conditional variances
is exactly $s=\sum_{t=0}^{t_1-1} (1-p)^{-2t}U_{t-1}'p(1-p)$.
By \eqref{diff} and \eqref{vsum} the ratio of $s$ to $n\rho/(1-\rho)$
converges to 1 in probability, as required for the martingale central limit theorem.
\end{proof}

To relate the distribution of $T_1$ to that of $X_{t_1}$ (or $\tX_{t_1}$)
we use the fact that $(X_t)$ has slope approximately
$-a$ near $t_1$; a similar argument was given by Martin-L\"of~\cite{ML86}.

\begin{lemma}\label{lstraight}
We have
\[
 \sup_{|t-t_1|\le t_0} |\tX_t-\tX_{t_1}-a(t_1-t)| = \op(\sigma_0).
\]
\end{lemma}
\begin{proof}
From \eqref{td2} we may write $\tX_t-\tX_{t_1}$ as
\[
 x_t-x_{t_1} +(1-p)^tS_t-(1-p)^{t_1}S_{t_1} = (f(t)-f(t_1)) + (1-p)^tS_t-(1-p)^{t_1}S_{t_1} +O(1).
\]
Recalling that $f'(t_1)=-a$ and $f''(t)=O(1/n)$ uniformly in $t$, the difference
between the first term and $a(t_1-t)$ is $O(|t-t_1|^2/n)=O(t_0^2/n) = o(\sigma_0)$.
For the rest, note that
\[
 |(1-p)^t-(1-p)^{t_1}|\le |1-(1-p)^{|t-t_1|}| \le p|t-t_1| \le pt_0.
\]
Since $S_{t_1}=\Op(\sqrt{t_1})$ and $pt_0\sqrt{t_1}=O(n^{-1}\omega\sigma_0\eps^{-1}\sqrt{\eps n})=o(\sigma_0)$,
it thus suffices to show that $\sup_{|t-t_1|\le t_0}|S_t-S_{t_1}|=\op(\sigma_0)$.
But this follows easily by applying Lemma~\ref{lmax} to the martingale
$(S_t-S_{t_1^-})_{t=t_1^-}^{t_1^+}$, which has final variance $O(t_0)=o(\sigma_0^2)$.
\end{proof}

\begin{proof}[Proof of Theorem~\ref{th_gen}.]
Recall from Lemma~\ref{l1} that $Z$, the number of components
explored by time $t_0$, satisfies $Z=\op(\sigma_0)$.
We have shown above that whp $T_1=\inf\{t:X_t=-Z-1\}$ lies between $t_1^-$ and $t_1^+$.
From \eqref{XXtclose}, $X_t$ is within $\op(\sigma_0)$ of $\tX_t$
at least until $T_1$.
It follows that at time $T_1$, we have $\tX_t=\op(\sigma_0)$.
Since $a=\Theta(\eps)$, Lemma~\ref{lstraight} thus gives
\begin{equation}\label{T1}
 T_1=t_1 +\tX_{t_1}/a +\op(\sigma_0/\eps).
\end{equation}
From Lemma~\ref{norm}, \eqref{td2} and the fact that $f(t_1)=0$, we have that $\tX_{t_1}$
is asymptotically normal with mean $0$ and variance 
\[
 (1-p)^{2t_1}n\rho/(1-\rho) \sim e^{-2\la\rho} n\rho/(1-\rho) = n\rho(1-\rho).
\]
Hence $\tX_{t_1}/a$ is asymptotically normal with mean 0 and variance
\[
 n\rho(1-\rho)/a^2 =\sigma^2.
\]
Since this variance is of order $\eps^{-1}n=\eps^{-2}\sigma_0^2$, the $\op(\sigma_0/\eps)$ error
term in \eqref{T1} is irrelevant, and $T_1$ is asymptotically normal
with mean $t_1=\rho n$ and variance $\sigma^2$. Finally, from Lemma~\ref{l1} we have $T_0=\op(\sigma_0/\eps)$.
It follows that $T_1-T_0$ is asymptotically normal with the parameters claimed in the theorem.

This shows the existence of a component with the claimed size. As shown by Nachmias and Peres~\cite{NP_giant},
it is easy to check that the rest of the graph corresponds to a subcritical random graph,
and whp will not contain a larger component.
\end{proof}

\begin{ack}
We are grateful to an anonymous referee for several suggestions improving the presentation
of the paper.
\end{ack}

\end{document}